\documentclass[12pt,reqno]{amsart}

\usepackage[arrow,matrix,curve]{xy}

\usepackage[dvips]{graphicx} 

\usepackage{amssymb, latexsym, amsmath, amscd, array,
}

\usepackage{amsfonts}     
\usepackage{tikz}     
\usepackage{scalefnt} 

\newtheorem{theorem}{Theorem}[section]
\newtheorem{lemma}[theorem]{Lemma}

\theoremstyle{definition}

\numberwithin{equation}{section}

\newcommand\N {{\mathbb N}} 

\newcommand\R {{\mathbb R}}
\newcommand\Q {{\mathbb Q}}
\newcommand\Z {{\mathbb Z}}

\newcommand\RRR{\mbox{I\!I\!R}}

\newcommand\NNN{\mbox{I\!I\!I\!\!N}}

\newcommand\Los{{\L}o{\'s}}

\begin{document}

\thispagestyle{empty}

\title [Toward a clarity of the extreme value theorem] {Toward a
clarity of the extreme value theorem}

\author{Karin U. Katz}

\author{Mikhail G. Katz}

\address{K. Katz, M. Katz, Department of Mathematics, Bar Ilan
University, Ramat Gan 52900 Israel} \email{katzmik@macs.biu.ac.il}

\author{Taras Kudryk}

\address{T. Kudryk, Department of Mathematics, Lviv National
University, Lviv, Ukraine}

\email{kudryk@mail.lviv.ua}



\bigskip\bigskip\bigskip\noindent
\begin{abstract}
We apply a framework developed by C.~S.~Peirce to analyze the concept
of \emph{clarity}, so as to examine a pair of rival mathematical
approaches to a typical result in analysis.  Namely, we compare an
intuitionist and an infinitesimal approaches to the extreme value
theorem.  We argue that a given pre-mathematical phenomenon may have
several aspects that are not necessarily captured by a single
formalisation, pointing to a complementarity rather than a rivalry of
the approaches.
\end{abstract}

\maketitle

\tableofcontents

\section{Introduction}

German physicist G. Lichtenberg (1742 - 1799) was born less than a
decade after the publication of the cleric George Berkeley's tract
\emph{The Analyst}, and would have certainly been influenced by it or
at least aware of it.  Lichtenberg wrote:
\begin{quote}
The great artifice of regarding small deviations from the truth as
being the truth itself [which the differential calculus is built
upon]%
\footnote{We have reinstated (in brackets) Lichtenberg's comment
concerning the differential calculus, which was inexplicably deleted
by Penguin's translator R. J.~Hollingdale.  The German original reads
as follows: ``Der gro{\ss}e Kunstgriff kleine Abweichungen von der
Wahrheit f\"ur die Wahrheit selbst zu halten, worauf die ganze
Differential Rechnung gebaut ist, ist auch zugleich der Grund unsrer
witzigen Gedanken, wo offt das Ganze hinfallen w\"urde, wenn wir die
Abweichungen in einer philosophischen Strenge nehmen w\"urden.''}
is at the same time the foundation of wit, where the whole thing would
often collapse if we were to regard these deviations in the spirit of
philosophic rigour (Lichtenberg 1765-1770 \cite[Notebook A, 1,
p.~21]{Lich}).
\end{quote}
The ``truth" Lichtenberg is referring to is that one cannot have
\[
(dx\not=0)\wedge(dx=0),
\]
namely infinitesimal errors are strictly speaking not allowed and,
philosophically speaking, equality must be true equality and only true
equality.  The allegation Lichtenberg is reporting as fact is
Berkeley's contention that infinitesimal calculus is based on a
logical iconsistency.  The ``deviation from the truth'' refers to the
so-called infinitesimal error that allegedly makes calculus possible.

\subsection{A historical re-appraisal}

Recently the underlying assumption of Lichtenberg's claim, namely the
inconsistency of the historical infinitesimal calculus, has been
challenged.  On Leibniz, see Knobloch (2002 \cite{Kn02}), (2011
\cite{Kn11}); Katz \& Sherry (2013 \cite{KS1}); Guillaume (2014
\cite{Gu}).  The present text continues the re-appraisal of the
history and philosophy of mathematical analysis undertaken in a number
of recent texts.  On Galileo, see Bascelli (2014 \cite{Ba14}).  On
Fermat, see Katz, Schaps \& Shnider (2013 \cite{KSS13}) and Knobloch
(2014 \cite{Kn14}).  On Euler, see Reeder (2012 \cite{Re}) and Bair et
al.~(2013 \cite{B11}), (2014 \cite{B12}).  On Cauchy, see Borovik \&
Katz (2012 \cite{BK}) and others.  Part of such re-appraisal concerns
the place of Robinson's infinitesimals in the history of analysis.

We will examine two approaches to the extreme value theorem, the
intuitionistic (i.e., relying on intuitionistic logic) and the
hyperreal, from the point of view of Peirce's \emph{three grades of
clarity} (Peirce \cite{Pei}).

It is sometimes thought that Robinson's treatment of infinitesimals
(see \cite{Ro66}) entails excursions into advanced mathematical logic
that may appear as baroque complications of familiar mathematical
ideas.  To a large extent, such baroque complications are unnecessary.
The 1948 ultrapower construction of the hyper-real fields by E.~Hewitt
\cite{Hew}, combined with J.~\Los's theorem \cite{Lo} from 1955 (whose
consequence is the transfer principle), provide a framework where a
large slice of analysis can be treated, in the context of an
infinitesimal-enriched continuum.  Such a viewpoint was elaborated by
Hatcher \cite{Ha82} in an algebraic context and by Lindstr\o{}m
\cite{Li88} in an analytic context (see Section~\ref{nine}).  We seek
to challenge the received wisdom that a modern infinitesimal approach
is necessarily baroque.

Laugwitz authored perceptive historical analyses of the work of Euler
and Cauchy (see e.g., \cite{Lau89}), but was somewhat saddled with the
numerous variants of the~$\Omega$-calculus that he developed jointly
with Schmieden (1958 \cite{SL}).  As a result, he did not fully stress
the simplicity of the ultrapower approach.  Synthetic Differential
Geometry (following Lawvere), featuring nilsquare infinitesimals, is
based on a category-theoretic framework; see Kock \cite{Ko06} and
J.~Bell \cite{Bel08}, \cite{Bel}.

\subsection{Practice and ontology}

To steer clear of presentism in interpreting classical
infinitesimalists like Leibniz and Euler, a crucial distinction to
keep in mind is that between syntax/procedures, on the one hand, and
semantics/ontology, on the other; see Benacerraf (1965 \cite{Be65})
for the dichotomy of practice \emph{vs} ontology.  Euler's inferential
moves involving infinitesimals and infinite numbers find good proxies
in the \emph{procedures} of Robinson's framework.  Meanwhile, the kind
of set-theoretic or type-theoretic ontology necessary to make
Robinson's framework acceptable to a modern rigorous mind is certainly
nowhere to be found in the classical infinitesimalists, any more than
Cantorian sets or Weierstrassian epsilontics.  

What we argue is \emph{not} that Robinson is the necessary logical
conclusion of Leibniz.  Rather, we have argued, with Felix Klein, that
there is a parallel B-track for the development of analysis that is
often underestimated in traditional A-track epsilontist
historiography.  The point is \emph{not} that the B-track leads to
Robinson, but rather than B-track is distinct from A-track; in fact, a
number of modern infinitesimal theories could potentially fit the bill
(for more details on the two tracks see Section~\ref{infi}).

The ultrapower construction is related to the classical construction
of the reals starting from the set~$\Q^\N$ of sequences of rationals,
in the following sense.  In the classical construction, one works with
the subset of~$\Q^\N$ consisting of Cauchy sequences.  Meanwhile, in
the ultrapower construction one works with {\em all\/} sequences (see
Section~\ref{infi}).  The researchers working in Robinson's framework
have often emphasized the need to learn model theory truly to
understand infinitesimals.  There may be some truth in this, but the
point may have been overstated.  A null sequence (i.e., sequence
tending to zero) generates an infinitesimal, just as it did in
Cauchy's \emph{Cours d'Analyse} (Cauchy 1821 \cite{Ca21}).  The link
between Cauchy's infinitesimals and those in the ultrapower
construction was explored by Sad, Teixeira \& Baldino \cite{STB} and
Borovik \& Katz~\cite{BK}.

In the sequel whenever we use the adjective \emph{intuitionistic}, we
refer to mathematical frameworks relying on intuitionistic logic
(rather than any specific system of \emph{Intuitionism} such as
Brouwer's).  Our intuitionistic framework is not BISH but rather that
of Troelstra and van Dalen \cite{TV}.  In a typical system relying on
intuitionistic logic, one finds counterexamples that seem
``paradoxical'' from the classical viewpoint, such as Brouwerian
counterexamples.  Such examples are present in Bishop's framework as
well, though the traditional presentations thereof by Bridges and
others tend to downplay this (see a more detailed discussion in
Section~\ref{four}).  Feferman noted that
\begin{quote}
[Bishop] finesses the whole issue of how one arrives at Brouwer's
theorem [to the effect that every function on a closed interval is
uniformly continuous] by saying that those are the only functions, at
least initially, that one is going to talk about (Feferman 2000
\cite{Fe00}).
\end{quote}
One of the themes of this text is the idea that, inspite of the
incompatibility between the two frameworks for EVT, important insight
about the problem can be gained from both.  Our approach is both
complementary and orthogonal to that of Wattenberg (1988 \cite{Wa}).

\section{Grades of clarity according to Peirce}
\label{one}

Is there anything unclear about the extreme value theorem (EVT)?  Or
rather, as an intuitionist might put it, is there anything
\emph{clear} about the classical formulation of the EVT?

The clarity referred to in our title alludes to the Peircean analysis
of the 3 grades (stages) of clarity in the emergence of a new concept.
The framework was proposed by C. S. Peirce in 1897, in the context of
his analysis of the concept of continuity and continuum, which, as he
felt at the time, is composed of infinitesimal parts (see
\cite[p.~103]{Ha}).  Peirce identified three stages in creating a
novel concept:

\begin{quote}
there are three grades of clearness in our apprehensions of the
meanings of words.  The first consists in the connexion of the word
with familiar experience \ldots The second grade consists in the
abstract definition, depending upon an analysis of just what it is
that makes the word applicable \ldots The third grade of clearness
consists in such a representation of the idea that fruitful reasoning
can be made to turn upon it, and that it can be applied to the
resolution of difficult practical problems (Peirce 1897 \cite{Pei}
cited by Havenel 2008 \cite[p.~87]{Ha}).
\end{quote}
The ``three grades'' can therefore be summarized as follows:
\begin{enumerate}
\item
familiarity through experience;
\item
abstract definition with an eye to future applications;
\item
fruitful reasoning ``made to turn'' upon it, with applications.
\end{enumerate}

The classical EVT asserts that a continuous function on a compact
interval has a maximum.  Attaining clarity in the matter of the EVT is
therefore contingent upon 
\begin{itemize}
\item
clarity in the matter of \emph{continuity}; and
\item
clarity in the matter of \emph{maximum}.
\end{itemize}
We will analyze these issues in Section~\ref{three}.

\section{Perceptual continuity}
\label{three}

Freudenthal \cite{Fr} wrote that Cauchy invented our notion of
continuity.  Yet 60 years earlier, Abraham Gotthelf Kaestner
(1719--1800) had the following to say in this matter:%
\footnote{Translation kindly provided by D.~Spalt.}
\begin{quote}
In a sequence%
\footnote{We have replaced the term {\em series\/} found in Kaestner,
by the term {\em sequence\/}, to conform with modern usage.}
of quantities the increase or the decrease takes place in accordance
with the law of continuity ({\em lege continui\/}), if after each term
of the sequence, the one that follows or preceeds the given term,
differs from it by as little as one wishes, in such a way that the
difference of two consecutive terms may be less than any given
quantity (Kaestner 1760 \cite[paragraph~322, p.~180]{Ka}).
\end{quote}

Similar formulations can be found even earlier in Leibniz.  Kaestner's
formulation of continuity is free of both infinitesimal language
and~$\epsilon, \delta$ terminology, yet it is a concise expression of
a perceptual idea of continuity, at stage 1 of the Peircean ladder.
Kaestner's definition is far from being a mathematically coherent one.
Bolzano and Cauchy retain priority even after this Kaestner is
discovered.  But what is interesting about his definition is its
\emph{local} nature.  Contrary to popular belief, Bolzano and Cauchy
were not the first to envision a local definition of continuity.

Decades later, Cauchy expressed perceptual continuity even more
succinctly, in terms of a function varying by {\em imperceptible
degrees\/} under minute changes of the independent variable (see,
e.g., his letter to Coriolis of 1837 cited by Bottazzini
\cite[p.~107]{Bot86}).

At the perceptual level, the EVT seeks to capture the intuitively
appealing idea of aiming for the ``highest peak on the graph".

Bolzano and Cauchy were the first ones to sense a need to provide
proof of assertions such as the intermediate value theorem (IVT).
Bolzano similarly proved the EVT in 1817 but his manuscript was only
rediscovered in the 1860s.  What kind of proof Cauchy (1789--1857)
might have provided for the EVT we will never know.  The EVT is
usually attributed to Weierstrass who proved it in a course during
1861 (see \cite[p.~18]{Gi} and \cite[p.~25, note~8]{Gi}).

The presence of potentially infinitely many points in the domain of
the function signals an immediate difficulty, residing in the
possibility of the graph rising higher and higher without bound.


One begins to perceive the complexity of the task with the realisation
of an inherent instability of an elusive ``highest point" on the
graph.  Namely, the values at a pair of far-away points may be so
close together as to make the choice difficult.
%
%
A minute perturbation of the function may have an appreciable effect
upon the answer.  We summarize the two difficulties signaled so far:
\begin{itemize}
\item
problem of existence of supremum over infinite domain;
\item
inherent instability of solution due to discontinuous dependence on
data.
\end{itemize}

A perceptual approach to looking for a solution would be to subdivide
the domain by means of a partition
\begin{equation}
\label{21b}
\{x_i\}
\end{equation}
so fine as to challenge the resolution of even the most powerful
physical microscope.  The number of partition points being finite,
there is necessarily one,
\[
x_{i_0},
\]
with the highest value of the function among the partition points.
Continuing the perceptual analysis, one now ``zooms in" on the atom
(or quark, or string) 
\[
c
\]
carrying the partition point~$x_{i_0}$.  Since, by hypothesis, the
function varies by imperceptible degrees, it does not deviate away
from the value~$f(x_{0})$ appreciably, on~$c$.

Of course, if the point is in a visible picture as seen by the human
eye, what looks like the maximum point may be one of several whose
height is visibly the same, but one does not know which of these, at a
higher magnification, will be the actual maximum.  Thus, one may be
zooming in on the wrong place.  Nonetheless, in the perceptual stage
being discussed here, one is {\em only\/} looking for what is visibly
a maximum.  Perceptually speaking, one will never know it's the wrong
place.

It is this phenomenon that makes it difficult to give a perceptual
proof that can be turned into a mathematical proof of the EVT.  If one
draws the curve, one may be able to see the maximum, but a mechanical
process of computing where the maximum is, that works without the
human eye, is more subtle, perhaps even impossible.

How is such a perceptual proof transformed into a mathematical
argument of a higher Peircean grade of clarity?

\section{Constructive clarity}

A constructive presentation of the EVT may be found in the text by
Troelstra and van Dalen \cite[p.~294-295]{TV}.  A majority of the
wider mathematical public is not intimately familiar with this
approach.  Therefore we will recall the main points in some detail.
When a step in a constructive argument is identical with the classical
one, we will preface it by an editorial clause ``as usual''.

The starting point is a challenge to the non-constructive nature of
``existence'' proofs in classical mathematics.  Such proofs generally
go under the name of {\em proof by contradiction\/}.  The main
ingredient in a proof by contradiction is the law of excluded middle
(LEM).

To provide an elementary example, consider the proof of irrationality
of the square root of~$2$, as discussed by E.~Bishop
\cite[p.~18]{Bi85}.  Constructively speaking, being \emph{irrational}
is a stronger property that simply being \emph{not rational}.  Namely,
\emph{irrationality} involves an explicit lower bound for the distance
from any rational (in terms of its denominator).  Thus, for each
rational~$m/n$, the integer~$2n^2$ is divisible by an odd power
of~$2$, while~$m^2$ is divisible by an even power of~$2$.  Hence we
have~$|2n^2-m^2|\geq 1$.  Here we have applied LEM (or more precisely
the law of trichotomy) to an effectively decidable predicate
over~$\Z$.  Since the decimal expansion of~$\sqrt{2}$ starts
with~$1.41\ldots$, we may assume~$\frac{m}{n} \leq 1.5$.  It follows
that
\begin{equation*}
|\sqrt{2} - \tfrac{m}{n}| = \frac{|2n^2 - m^2|}{n^2 \left(
  \sqrt{2}+\tfrac{m}{n} \right)} \geq \frac{1}{n^2
  \left(\sqrt{2}+\tfrac{m}{n}\right)} \geq \frac{1}{3n^2},
\end{equation*}
yielding a numerically meaningful proof of irrationality which avoids
the use of LEM in its classical form (see Section~\ref{25} for a more
advanced discussion).  This is, of course, a special case of
Liouville's theorem on diophantine approximation of algebraic numbers;
see \cite{HW}.

The intuitionist/constructivist challenge to classical mathematics has
traditionally targeted the logical principle expressed by LEM.  An
additional point to keep in mind is a parallel change in the
interpretation of the existential quantifier~$\exists$, which now
takes on a constructive meaning.  Thus, to show that
\[
\exists x : P(x),
\]
one needs to specify an effective procedure for exhibiting such
an~$x$.

We say that~$x$ is the {\em supremum\/} of a set~$X\subset \R$ if and
only if the following condition is satisfied:
\begin{equation*}
(\forall y \in X) \;\; (y\leq x) \; \wedge \; (\forall k\in \N) \;
(\exists y\in X) \left( y> x - 2^{-k} \right).
\end{equation*}
%
As usual, the supremum of a function~$f$ on~$[a,b]$ is the sup of the
set~$\{ f(x): x\in [a,b] \}$.  A set~$X\subset \R$ is {\em totally
bounded\/} if for each~$k\in \N$, there is a finite collection of
points~$x_0, \ldots , x_{n-1} \in X$ such that
\begin{equation*}
(\forall x \in X) \; (\exists i < n) \left( | x - x_i | < 2^{-k}
\right).
\end{equation*}

\begin{lemma}
\label{23}
A totally bounded set~$X\subset \R$ has a supremum.
\end{lemma}

\begin{proof}
The proof exploits a constructive version of the notion of a Cauchy
sequence.  Namely, one requires~$| x_k - x_{k+n} | < 2^{-k}$ for
every~$n\in \N$ (here~$n$ is independent of~$k$).  For each~$k$ we
specify a corresponding finite collection of points~$x_{k,0}, \ldots ,
x_{k,n-1} \in X$ (with~$n=n(k)$) such that~$(\forall x\in X)\;(\exists
i<n(k))\;\left(|x-x_{k,i}|<2^{-k} \right)$.  Next, let
\[
x_k:= \max \{ x_{k,i} : i < n(k) \}.
\]
It follows that
\begin{equation}
\label{21}
(\forall x \in X) \; (\forall k) \; \left( x - x_k < 2^{-k} \right),
\end{equation}
and we apply formula \eqref{21} to complete the proof.%
\footnote{In more detail, formula \eqref{21} applied to~$x=x_{k+1}$
implies~$x_{k+1}-x_k<2^{-k}$.  Applying the same formula with~$k+1$ in
place of~$k$ to~$x=x_k$ implies~$x_{k}-x_{k+1}<2^{-k-1}$.
Hence~$|x_k-x_{k+1}|<2^{-k}$.  Summing a geometric series we
obtain~$|x_k-x_{k+n}|<2^{-k}+2^{-k-1}+\ldots+2^{-k-n-1}<2^{-k+1}$.
The sequence~$\langle x_n : n\in \N \rangle$ is therefore Cauchy.  One
easily shows that its limit equals~$\sup X$.}
\end{proof}

Compared with the classical version of the EVT, the constructive
version requires a stronger hypothesis of {\em uniform\/} continuity,%
\footnote{Pointwise continuity and uniform continuity on a compact
interval are equivalent with respect to classical logic, but not with
respect to intuitionistic logic.}
and yields a weaker result, namely the existence of~$\sup(f)$ but not
the existence of a maximum.  Of course, the advantage of the
constructive version is a clarification of the nature of
constructively provable results, and may be considered more
``honest''.

\begin{theorem}[Constructive EVT]
\label{24}
Let~$f : [0,1] \to \R$ be uniformly continuous.  Then~$\sup(f)$
exists.
\end{theorem}

\begin{proof}
Let~$\alpha$ be a modulus for~$f$, so that one has
\[
(\forall x, y \in [0,1]) \; (\forall k) \; \left( |x-y| < 2^{-\alpha
k} \implies |f(x)-f(y)| < 2^{-k} \right).
\]
Then the set~$\{ f(x): x\in [0,1] \}$ is totally bounded.  Indeed,
let~$n> 2^{\alpha k}$.  Let~$x_i:= \frac{i}{n}$ for~$i<n$.  Then by
uniform continuity, we obtain that~$(\forall x\in[0,1])\;(\exists
i<n)\left(|f(x)-f(x_i)|<2^{-k}\right)$, and therefore the set of
values is totally bounded.  The proof is completed by applying
Lemma~\ref{23}.
\end{proof}

The constructive impossibility of strengthening the conclusion of the
theorem is discussed in Section~\ref{four}.

\section{Counterexample to the existence of a maximum}
\label{four}
\label{25}

The conclusion of Theorem~\ref{24} cannot be strengthened to the
existence of a maximum, in the sense of the constructive formula
\[
(\exists x \in [0,1]) \; \left( f(x)=\sup(f) \right).
\]
Indeed, let~$a$ be any real such that~$a\leq 0$ or~$a\geq 0$ is
unknown.  Next, define~$f$ on~$[0,1]$ by setting~$f(x)=ax$.
Then~$\sup(f)$ is simply~$\max(0, a)$, but the point where it is
attained cannot be captured constructively (see~\cite[p.~295]{TV}).
To elaborate on the foundational status of this example, note that the
law of excluded middle:
\[
P \vee \neg P
\]
(``either~$P$ or (not~$P$)''), is the strongest principle rejected by
constructivists.  A weaker principle is the LPO (limited principle of
omniscience).  The LPO is the main target of Bishop's criticism in
\cite{Bi75}.  The LPO is formulated in terms of sequences, as the
principle that it is possible to search ``a sequence of integers to
see whether they all vanish'' \cite[p.~511]{Bi75}.  The LPO is
equivalent to the law of trichotomy:
\[
(a< 0) \vee (a=0) \vee (a> 0).
\]
An even weaker principle is~$(a\leq 0) \vee (a \geq 0)$.  The
existence of~$a$ satisfying the negation
\[
\neg \left( (a\leq 0) \vee (a \geq 0) \right)
\]
is exploited in the construction of the counterexample under
discussion.  The principle~$(a\leq 0) \vee (a \geq 0)$ is false
intuitionistically.  After discussing real numbers~$x\geq 0$ such that
it is ``{\em not\/}'' true that~$x>0$ or~$x=0$, Bishop writes:
\begin{quotation}
In much the same way we can construct a real number~$x$ such that it
is {\em not\/} true that~$x\geq 0$ or~$x\leq 0$ (Bishop 1967
\cite[p.~26]{Bi67}), (Bishop \& Bridges 1985 \cite[p.~28]{BB}).
\end{quotation}

The fact that an~$a$ satisfying~$\neg ((a\leq 0) \vee (a \geq 0))$
yields a counterexample~$f(x)=ax$ to the extreme value theorem (EVT)
on~$[0,1]$ is alluded to by Bishop in \cite[p.~59, exercise~9]{Bi67};
\cite[p.~62, exercise~11]{BB}.  Bridges interprets Bishop's italicized
``{\em not\/}'' as referring to a Brouwerian counterexample, and
asserts that trichotomy as well as the principle~$(a\leq 0) \vee (a
\geq 0)$ are independent of Bishopian constructivism.  See
Bridges~\cite{Br94} for details; a useful summary may be found in
Taylor~\cite{Ta}.

\section{Reuniting the antipodes}

The advantage of the constructive framework is the anchoring of the
distinction between what can be exhibited and what cannot, in the very
mathematical formalism.  Rather than being an afterthought that may or
may not trickle down to the students, the distinction is built into
the intuitionistic hardware.  Meanwhile, the level of detail required
to operate such machinery risks masking the simple perceptual insights
at the level of the cognitive underpinnings of the EVT.

Viewed as a {\em companion\/} to the classical approach, the
intuitionistic framework can usefully enhance constructive issues that
are otherwise relegated to footnotes, appendices, or optional material
in textbooks.  Meanwhile, viewed as an {\em alternative\/} to the
classical approach, the intuitionistic framework risks masking
important conceptual phenomena available in classical idealisations,
particularly in areas such as geometry and mathematical physics.
Minimal surfaces, geodesics, variational principles, etc., are
inextricably tied in with ever more sophisticated implementations of
the classical EVT, and rely upon the existence of the actual extremal
points rather than merely suprema.  Thus, the existence of the
Calabi-Yau manifolds, ubiquitous in both differential geometry and
mathematical physics, is nonconstructive; see Yau \& Nadis \cite{YN}.

More specifically, general relativity routinely exploits versions of
the extreme value theorem, in the form of the existence of solutions
to variational principles, such as geodesics, be it spacelike,
timelike, or lightlike.  At a deeper level, S.P.~Novikov \cite{No1,
No2} wrote about Hilbert's meaningful contribution to relativity
theory, in the form of discovering a Lagrangian for Einstein's
equation for spacetime.  Hilbert's deep insight was to show that
general relativity, too, can be written in Lagrangian form, which is a
satisfying conceptual insight.

A radical constructivist's reaction would be to dismiss the material
discussed in the previous paragraphs as relying on LEM (needed for the
EVT), hence lacking numerical meaning, and therefore meaningless.  In
short, radical constructivism (as opposed to the liberal variety)
adopts a theory of meaning amounting to an ostrich effect as far as
certain significant scientific insights are concerned.  A quarter
century ago, M.~Beeson already acknowledged constructivism's problem
with the calculus of variations in the following terms:
\begin{quote}
Calculus of variations is a vast and important field which lies right
on the frontier between constructive and non-constructive mathematics
(Beeson 1985 \cite[p.~22]{Bee}).
\end{quote}
An even more striking example is the Hawking-Penrose singularity
theorem explored by Hellman \cite{He98}, which relies on fixed point
theorems and therefore is also constructively unacceptable, at least
in its present form.  However, the singularity theorem does provide
important scientific insight.  Roughly speaking, one of the versions
of the theorem asserts that certain natural conditions on curvature
(that are arguably satisfied experimentally in the visible universe)
force the existence of a singularity when the solution is continued
backward in time, resulting in a kind of a theoretical justification
of the Big Bang.  Such an insight cannot be described as
``meaningless'' by any reasonable standard of meaning preceding
nominalist commitments; see (Katz \& Katz 2012 \cite{KK11a}) for more
details.

\section{Kronecker and constructivism}

Kronecker subdivided mathematics into three fields: analysis, algebra,
and number theory (or arithmetic).  In his 1861 inaugural speech at
the Academy of Science of Berlin, Kronecker said: 
\begin{quote}
The study of complex multiplication of elliptic functions leading to
works the object of which can be characterized as being drawn from
analysis, motivated by algebra and driven by number theory (see
Gauthier \cite[p.~39]{Ga02}).
\end{quote}
In his 1886 letter to Lipschitz, he declared that with the publication
of his 1882 work \emph{Grundz\"uge einer arithmetischen Theorie der
algebraischen Gr\"ossen}, he has found the long-sought foundations of
his theory of forms with the arithmetisation of algebra which had been
the goal of his mathematical life.


In his 1891 lectures \cite{Kro}, Kronecker criticized Bolzano for
having used the crudest means (\emph{mit den rohesten Mitteln}) in his
proof of the intermediate value theorem; see Gauthier (2009
\cite[p.~225]{Ga09}), (2013 \cite[p.~39]{Ga13}); Boniface \&
Schappacher \cite[p.~269-270]{BS}).


Thus, Kronecker did not consider geometry and mathematical physics as
part of mathematics.  Of the three fields of analysis, geometry, and
mathematical physics (see Boniface \& Schappacher \cite[p.~211]{BS}),
he argued in favor of a constructive arithmetisation of analysis
alone.%
\footnote{The meaning of the term ``arithmetisation'' has changed
since Kronecker introduced it, and today refers to the traditional
set-theoretic foundations of~$\R$ in particular and analysis in
general.}
This was to be a reformulation of analysis in terms of the natural
numbers rather than the speculative constructs such as~$\R$, as
envisioned by his contemporaries Cantor, Dedekind, and Weierstrass.
Kronecker plainly acknowledged that two-thirds of what is today
considered mathematics (namely, what he referred to as geometry and
physics) is not amenable to such arithmetisation.%
\footnote{H.~Edwards, a contemporary expert on Kronecker's thinking,
has done much to restore balance in the popular perception of
Kronecker (see e.g., \cite{Ed}), but may have overlooked the
three-fold partition and its ramifications for constructivisation.}
Radical constructivists attempting to enlist Kronecker to their cause
of constructivizing \emph{all} of modern mathematics may therefore be
``more vigorous than accurate'', to quote Robinson (1968 \cite{Ro68}).

As a transition to the remainder of this text, we note that classical
mathematics can be thought of as an extension of constructive
mathematics, inasmuch as the former uses more and stronger axioms than
the latter, although the verificationist interpretation of the
existence quantifier necessarily leads to a clash with the classical
viewpoint (see end of Section~\ref{four} for details).  Meanwhile,
analysis over the hyperreals can be done in the framework of the
standard Zermelo-Fraenkel axiom system with the axiom of choice.  The
key consequence of the latter is the existence of ultrafilters proved
by Tarski \cite{Tar} in 1930 (the use of ultrafilters is explained in
Subsection~\ref{ultrapower}).

An admirable attempt to bridge the gap between ``the constructive and
nonstandard views of the continuum'' resulted in a volume of
publications \cite{Sc}, but not immediately in a unity of purpose.  A
type of proof described as ``constructive modulo an ultrafilter''
proposed by Ross \cite{Ros01}, \cite{Ros06} represents an interesting
attempt at a dialogue.  Our theme here is that, even though the two
conceptual frameworks may not be compatible, both can yield useful
mathematical insight.

An intriguing proposal for bridging the gap specifically in the
context of the extreme value theorem was recently made by
Schuster~\cite{Sc10}.  Here one exploits the principle of unique
choice, alternatively called the principle of non-choice.  The
heuristic idea is that, if uniqueness is sufficiently ubiquitous
(e.g., ``if a continuous function on a complete metric space has
approximate roots and, in a uniform manner, at most one root''), then
existence follows, as well (``it actually has a root'') (see also
\cite{Sch06}).  There is a vast literature on the related uniqueness
paradigm.  Thus, in Kohlenbach \cite{Ko93}, the uniform notion of
uniqueness was introduced under the name ``modulus of (uniform)
uniqueness'' (see also his book \cite{Ko}).  Some of these ideas were
anticipated in a 1979 text by Kreinovich \cite{Kr79} (see also his
review \cite{Kr82}).  The theorem that if a computable function in
$C[0,1]$ attains its maximum at a unique point then that point is
computable, was already proved by Grzegorczyk~\cite{Grz} in 1955.
Further work includes Ishihara \cite{Is}; Berger and
Ishihara~\cite{Ber05}; Berger, Bridges, and Schuster \cite{Ber06};
Diener and Loeb~\cite{DL09}; Schwichtenberg~\cite{Schw};
Bridges~\cite{Br10}; and others.

\section{Infinitesimal clarity}
\label{infi}

Approximation issues stressed in the constructive approach are
important for both mathematics and its applications.  Similarly, it is
important for the student to realize that the tangent line is the
limit of secant lines, which gives geometric motivation for the idea
of derivative.

\subsection{Nominalistic reconstructions}

At the same time, it is important to realize that the Weierstrassian
$(\epsilon,\delta)$ approach tends to remove motion and geometry from
the definition of basic concepts of the calculus.  The game of ``you
tell me the epsilon, I will tell you the delta'' superficially
resembles approximation theory.  However, the crucial issue is the
functional dependence of $\delta$ on $\epsilon$, rather than any
specific approximation; no wonder engineering students, who are
certainly vitally concerned with approximation, are seldom taught the
$(\epsilon, \delta)$ method.  In reality dressing students to perform
multiple-quantifier epsilontic logical stunts on pretense of teaching
them infinitesimal calculus is merely a way of dressing up a bug to
look like a feature (to borrow a quip from computer science folklore),
as is apparent if one compares this approach to the lucidity of its
infinitesimal counterpart.

The Weierstrassian $(\epsilon,\delta)$ approach was necessitated by an
inability to justify the ontological material that naturally arose in
scientific inquiry during the 17th, 18th, and 19th centuries, namely
the infinitesimals, which were present at the conception of the theory
by Leibniz, Johann Bernoulli, and others.  For more details, see Bair
et al.~\cite{B11}; Bottazzi et al.~\cite{Bo14}; Mormann et
al.~\cite{MK}.

Starting in the 1870s, Cantor, Dedekind, and Weierstrass justified the
logical complications they introduced into the foundations, in terms
of their success in eliminating foundational material they were unable
to justify.  Cantor went as far as calling infinitesimals the
``cholera bacillus of mathematics", and published a paper purporting
to ``prove" that they are self-contradictory; see Ehrlich (2006
\cite{Eh06}) for more details.  Their work amounted to a nominalistic
reconstruction of analysis, by eliminating ontological material they
could not account for; see Katz \& Katz (2012 \cite{KK11a}) for more
details.  The success of their reconstruction resulted in a widespread
ad-hoc ontological commitment to an exclusive reality of the real
numbers (or their constructive analogues).

Such a reductive philosophical commitment has taken a toll on the
development of mathematics.  Thus, Cauchy's Dirac delta function and
its applications in Fourier analysis were forgotten for over a
century, because of the reductive ideology that eliminated
infinitesimals, without which Cauchy's applications of what would
later be called the Dirac delta function could not be sustained; see
Laugwitz~(1989 \cite{Lau89}); Katz \& Tall (2013 \cite{KT}).
Similarly, Cauchy was the one who invented our notion of continuity,
and he defined it in terms of infinitesimals.  To Cauchy, an
infinitesimal was generated by a null sequence.  This is related both
perceptually and formally to the construction of infinitesimals in the
ultrapower approach (see \cite{STB}).

An infinitesimal-enriched continuum offers a possibility of mimicking
more closely the perceptual analysis of the EVT, in constructing a
formal proof, due to the availability of a hierarchical number system,
with an Archimedean continuum (A-continuum for short) englobed inside
an infinitesimal-enriched Bernoullian continuum (B-continuum for
short).

\subsection{Klein on rivalry of continua}

Felix Klein described a rivalry of such continua in the following
terms.  Having outlined the developments in real analysis associated
with Weierstrass and his followers, Klein pointed out that
\begin{quote}
The scientific mathematics of today is built upon the series of
developments which we have been outlining.  But {\em an essentially
different conception of infinitesimal calculus has been running
parallel with this [conception] through the centuries\/} (Klein
\cite[p.~214]{Kl}) [emphasis added--authors].
\end{quote}
Such a different conception, according to Klein,
\begin{quote}
harks back to old metaphysical speculations concerning the {\em
structure of the continuum\/} according to which this was made up of
[...] infinitely small parts \cite[p.~214]{Kl} [emphasis
added---authors].
\end{quote}
Victor J. Katz (2014 \cite{Ka14}) appears to imply that a B-track
approach based on notions of infinitesimals or indivisibles is limited
to ``the work of Fermat, Newton, Leibniz and many others in the 17th
and 18th centuries''.  This does not appear to be Klein's view.  Klein
formulated a condition, in terms of the mean value theorem, for what
would qualify as a successful theory of infinitesimals, and concluded:
\begin{quote}
I will not say that progress in this direction is impossible, but it
is true that none of the investigators have achieved anything positive
(Klein 1908 \cite[p.~219]{Kl}).
\end{quote}
Klein was referring to the current work on infinitesimal-enriched
systems by Levi-Civita, Bettazzi, Stolz, and others.  In Klein's mind,
the infinitesimal track was very much a current research topic; see
Ehlrich (2006 \cite{Eh06}) for a detailed coverage of the work on
infinitesimals around~1900.

\begin{figure}
\[
\xymatrix@C=95pt{{} \ar@{-}[rr] \ar@{-}@<-0.5pt>[rr]
\ar@{-}@<0.5pt>[rr] & {} \ar@{->>}[d]^{\hbox{st}} & \hbox{\quad
B-continuum} \\ {} \ar@{-}[rr] & {} & \hbox{\quad A-continuum} }
\]
\caption{Taking standard part}
\label{31}
\end{figure}
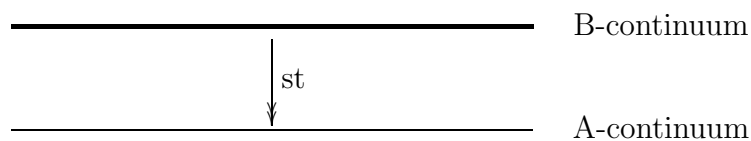

\subsection{Formalizing Leibniz}

Leibniz's approach to the differential quotient%
\footnote{See (Sherry 1987 \cite{She87}) and (Katz \& Sherry 2012
\cite{KS1}, \cite{KS2}).}
\[
\frac{d y}{d x}
\] 
(today called the derivative) was formalized by Robinson.  Here one
exploits a map called \emph{the standard part}, denoted~``st'', from
the finite part of a B-continuum, to the A-continuum, as illustrated
in Figure~\ref{31}.

In the context of the hyperreal extension of the real numbers, the map
st ``rounds off'' each finite hyperreal~$x$ to the nearest real
$x_0=\text{st}(x)\in \R$.  In other words, the map ``st'' collapses
the cluster of points infinitely close to a real number~$x_0$, back
to~$x_0$.  Note that both the term ``hyper-real field'', and an
ultrafilter construction thereof, are due to E.~Hewitt in 1948 (see
\cite[p.~74]{Hew}).  The transfer principle allowing one to extend
every first-order real statement to the hyperreals, is due to
J.~\Los{} in 1955 (see \cite{Lo}).  Thus, the Hewitt-\Los{} framework
allows one to work in a B-continuum satisfying the transfer principle.
See (Keisler 1994 \cite{Kei}) for a historical outline.

We will denote such a B-continuum by the new symbol~\RRR.  We will
also denote its finite (limited) part by
\[
\RRR_{<\infty} .
\]
The map ``st'' sends each finite point~$x\in \RRR$, to the real point
st$(x)\in \R$ infinitely close to~$x$:
\[
\xymatrix{\quad \RRR_{{<\infty}}^{~} \ar@{->>}[d]^{{\rm st}} \\ \R}
\]
We illustrate the construction by means of an infinite-resolution
microscope in Figure~\ref{tamar}.

\def\lineticktop#1#2{
  \draw[-] ({#1},3.075) node [anchor=south] {#2} -- (#1,2.925) ;
}
 
\def\linetickbot#1#2{
  \draw[-] ({#1},-0.075) node [anchor=north] {#2} -- (#1,0.075) ;
}
 
\def\microscope#1{
  \draw [thick,-         ] (#1 -0.25,2.15) -- (#1+0,3) -- (#1 +0.25,2.15) ;
  \draw [thick,fill=white] (#1 +0.00,1.50) circle (1) ;  
  \fill [white           ] (#1 -0.25,2.15) -- (#1+0,3) -- (#1 +0.25,2.15) ;
  \draw [thick,blue,->   ](#1-0.95,1.5) -- (#1+0.95,1.5) ;
}
 
\def\microtick#1#2#3#4#5#6{
  \draw [-,#5] (#1+0.375*#2,1.5+#3) node [anchor=#6,font=\small] {#4} -- (#1+0.375*#2,1.5-#3) ;
  \draw [red,thin] (#1+0.375*#2,1.3) to [out=-90,in=90] (#1+0.00,0.6);
}
 
\begin{figure}

\begin {tikzpicture} [scale=2]
  \draw[blue,thick,->]      (-1.1,3) -- (4.1,3) coordinate (x axis);
 
  \microscope    {1.41}
  \microtick     {1.41}{-0.2}{0.1}{$r$}{}{south}
  \microtick     {1.41}{ 0.7}{0.05}{\tiny $r+\beta$}{blue}{south}
  \microtick     {1.41}{ 2.1}{0.05}{\tiny $r+\gamma$}  {blue}{south}
  \microtick     {1.41}{-1.7}{0.05}{\tiny $r+\alpha$}    {blue}{south}
  \draw[red,->] (1.41,0.6) -- (1.41,0.325) node [anchor=west] {$\scriptstyle \operatorname{st}$} -- (1.41,0.15) ;
 
  \lineticktop{-1}{$-1$}
  \linetickbot{-1}{$-1$}
  \lineticktop{ 0}{$ 0$}
  \linetickbot{ 0}{$ 0$}
  \lineticktop{+1}{$ 1$}
  \linetickbot{+1}{$ 1$}
  \lineticktop{+2}{$ 2$}
  \linetickbot{+2}{$ 2$}
  \lineticktop{+3}{$ 3$}
  \linetickbot{+3}{$ 3$}
  \lineticktop{+4}{$ 4$}
  \linetickbot{+4}{$ 4$}
 
  \lineticktop{+1.41}{}
  \linetickbot{+1.41}{$ r$}
 
  \draw[->]      (-1.1,0) -- (4.1,0) coordinate (x axis);
\end {tikzpicture}

\caption{\textsf{The standard part function, st, ``rounds off" a
finite hyperreal to the nearest real number.  The function st is here
represented by a vertical projection.  An ``infinitesimal microscope"
is used to zoom in on an infinitesimal neighborhood of a standard real
number $r$, where $\alpha$, $\beta$, and~$\gamma$ represent typical
infinitesimals.  Courtesy of Wikipedia.}}
\label{tamar}
\end{figure}

Robinson defined the derivative as $\hbox{st} \left( \frac{\Delta
y}{\Delta x} \right)$, instead of~$\Delta y/\Delta x$.  For an
accessible exposition (see H.~J.~Keisler \cite{Ke, Kei}).%
\subsection{Ultrapower}
\label{ultrapower}

To elaborate on the ultrapower construction of the hyperreals,
let~$\Q^\N$ denote the space of sequences of rational numbers.
Let~$\left( \Q^\N \right)_C$ denote the subspace consisting of Cauchy
sequences.  The reals are by definition the quotient field~$\R:=
\left. \left( \Q^\N \right)_C \right/ \mathcal{F}_{\!n\!u\!l\!l}$,
where~$\mathcal{F}_{\!n\!u\!l\!l}$ contains all null sequences.
Meanwhile, the hyperreals can be obtained by forming the
quotient~$\RRR= \left.  \R^\N \right/ \mathcal{F}_{u}$, where a
sequence~$\langle u_n \rangle$ is in~$\mathcal{F}_{u}$ if and only if
the set~$\{ n \in \N : u_n = 0 \}$ is a member of a fixed ultrafilter.
To give an example, the
sequence~$\left\langle\tfrac{(-1)^n}{n}:n\in\N\right\rangle$
represents a nonzero infinitesimal, whose sign depends on whether or
not the set~$2\N$ is a member of the ultrafilter.  A helpful
``semicolon'' notation for presenting an extended decimal expansion of
a hyperreal was described by Lightstone~\cite{Li}.

\section{Hyperreal extreme value theorem}

We now return to the EVT.  The classical proof of the EVT usually
proceeds in two or more stages.  Typically, one first shows that the
function is bounded.  Then one proceeds to construct an extremum by
one or another procedure involving choices of sequences.  The
hyperreal approach is both more economical (there is no need to prove
boundedness first) and less technical.  To show that a continuous
function~$f(x)$ on~$[0,1]$ has a maximum, let
\[
H
\]
be an infinite hypernatural number (for instance, the one represented
by the sequence~$\langle 1,2,3,\ldots\rangle$ with respect to the
ultrapower construction outlined in Subsection~\ref{ultrapower}).  The
real interval~$[0, 1]$ has a natural hyperreal extension.  Consider
its partition into~$H$ subintervals of equal infinitesimal
length~$\tfrac{1}{H}$, with partition points~$x_i = \tfrac{i}{H}$
as~$i$ ``runs'' from~$0$ to~$H$.  The existence of such a partition
follows by the transfer principle (see more below) applied to the
first order formula
\[ 
(\forall n\in \N) \; (\forall x\in [0,1]) \; ( \exists i<n ) \; \left(
\tfrac{i}{n} \leq x < \tfrac{i+1}{n} \right).
\]
The function~$f$ is naturally extended to the hyperreals between~$0$
and~$1$.  Note that in the real setting (when the number of partition
points is finite), a point with the maximal value of~$f$ can always be
chosen among the partition points~$x_i$, by induction.  We have the
following first order property expressing the existence of a maximum
of~$f$ over a finite collection:
\[
(\forall n\in \N) \; (\exists i_0<n) \; ( \forall i<n ) \left(
f(\tfrac{i_0}{n}) \geq f(\tfrac{i}{n}) \right) .
\]
We now apply the transfer principle to obtain
\begin{equation}
\label{41}
(\forall H \in \NNN) \; ( \exists i_0<H ) \; ( \forall i<H ) \; \left(
f(\tfrac{i_0}{H}) \geq f(\tfrac{i}{H}) \right),
\end{equation}
where~$\NNN$ is the collection of hypernatural numbers.
Formula~\eqref{41} is true in particular for a particular infinite
hypernatural~$H\in \NNN \setminus\N$.  Thus, there is a
hypernatural~$i_0$ such that~$0\leq i_0\leq H$ and
\begin{equation}
\label{42}
f(x_{i_0})\geq f(x_i)
\end{equation}
for all~$i = 0, \ldots, H$.  Consider the real point~$c={\rm
st}(x_{i_0})$ where ``st'' is the standard part function.  By
continuity of~$f$, we have~$f(x_{i_0})\approx f(c)$, and
therefore
\[
{\rm st}(f(x_{i_0}))= f({\rm st} (x_{i_0}))=f(c).
\]
An arbitrary real point~$x$ lies in a suitable sub-interval of the
partition, namely~$x\in[x_i,x_{i+1}]$, so that st$(x_i)=x$.  Applying
``st'' to inequality~\eqref{42}, we obtain
\[
{\rm st}(f(x_{i_0}))\geq {\rm st}(f(x_i)).
\]
Hence~$f(c) \geq f(x)$, for all real~$x$, proving~$c$ to be a maximum
of~$f$.  Note that the argument follows closely our \emph{perceptual}
analysis in Section~\ref{three}.

\section{Approaches and invitations}
\label{nine}

Robinson's approach in (Robinson 1966 \cite{Ro66}) was formulated in
the framework of model theory of mathematical logic.

Two decades later, Lindstr\o{}m (1988 \cite{Li88}) presented an
alternative analytical approach in his text \emph{An invitation to
nonstandard analysis}, described as follows:
\begin{quote}
I have tried to make the subject look the way it would had it been
developed by analysts or topologists and not logicians.  This is the
explanation for certain unusual features such as my insistence on
working with ultrapower models and my willingness to downplay the
importance of first order languages (Lindstrom \cite[p.~1]{Li88}).
\end{quote}
In this approach, the proof of EVT, though essentially the same, is
even more elementary and short, because one does not need to use the
transfer principle.

\section{Conclusion}

The strength of the constructive approach is the ability to place into
sharp relief a hard-nosed analysis of what can be effectively
exhibited, and what cannot.  The strength of the infinitesimal
approach is its closer fit with the perceptual analysis of the
phenomenon at the heart of the EVT.  The complementarity of the
resulting insights ultimately points to a companionship, rather than a
rivalry, between the two approaches.  Such fruitful complementarity
persists inspite of possible formal incompatibilities of the
intuitionistic and the classical frameworks.

\end{document}